\documentclass[11 pt]{amsart}

\usepackage{lineno,hyperref}
\usepackage{amsmath,amssymb}
\usepackage{enumitem}
\usepackage{lineno}
\newtheorem{thm}{Theorem}[section]
\newtheorem{defn}[thm]{Definition}
\newtheorem{prop}[thm]{Proposition}
\newtheorem{cor}[thm]{Corollary}
\newtheorem{lem}[thm]{Lemma}

\begin{document}


\title{On wovenness of $K$-fusion frames}

\author[A. Bhandari]{Animesh Bhandari $^\dagger$}

\address{$^\dagger$Department of Mathematics\\ NIT Meghalaya\\ Shillong 793003\\ India}
\email{animesh@nitm.ac.in}

\author[S. Mukherjee]{Saikat Mukherjee {$^{\dagger *}$}}
\address{$^\dagger$Department of Mathematics\\ NIT Meghalaya\\ Shillong 793003\\ India}
\email{saikat.mukherjee@nitm.ac.in}

$\thanks{* Corresponding author; Supported by DST-SERB project MTR/2017/000797}$


\begin{abstract}
In frame theory literature, there are several generalizations of frame, $K$-fusion frame presents a flavour of one such generalization, basically it is an intertwined replica of $K$-frame and fusion frame.
$K$-fusion frames come naturally (having significant applications) when one needs to reconstruct functions (signals) from a large data in the range of a bounded linear operator. Getting inspiration from the concept of weaving frames in Hilbert space, we study the weaving form of $K$-fusion frames which have significant applications in wireless sensor networks. This article produces various characterizations of weaving $K$-fusion frames in different spaces. Furthermore, Paley-Wiener type perturbation and conditions on erasure of frame components have been assembled to scrutinize woven-ness of the same.\\

\end{abstract}

\maketitle

\section{Introduction}\label{Sec-Preli}
The notion of Hilbert frames was first introduced by Duffin and Schaeffer \cite{DuSc52} in 1952. After several decades, in 1986, frame theory has been popularized by the groundbreaking work by Daubechies, Grossman and Meyer \cite{DaGrMa86} by showing its practical significance in distributed signal processing. Since then frame theory has been widely applicable by mathematicians and engineers in various fields.


Furthermore, frame theory literature became  familiarized through several generalizations, one such generalization is $K$-fusion frame, $K$-fusion frame was first studied by Liu et al. \cite{LiuLi18}. After that  Neyshaburi et al. \cite{NeAr18} and Bhandari at al. \cite{Bh17} produced several characterizations of $K$-fusion frame.



Throughout the paper $\mathcal H$ is a separable Hilbert space, $\mathcal{L}(\mathcal{H}_1, \mathcal{H}_2)$ the space of all bounded linear operators from  $\mathcal{H}_1$ into $\mathcal{H}_2$,  $\mathcal L(\mathcal H)$ for $\mathcal L(\mathcal H, \mathcal H)$, $P_{\mathcal A}$ is the orthogonal projection on $\mathcal A$, $\mathcal I$ is countable index set, $R(T)$ is denoted as range of a bounded linear operator $T$ and $T^{\dag}$ is the Moore-Penrose pseudo inverse of $T$.

\begin {defn} ({\bf $K$-Fusion Frame})
Let $K \in \mathcal L(\mathcal H)$ for which a weighted collection  $\mathcal W_w = \{ (\mathcal W_i, w_i) \}_{i \in \mathcal I}$ of closed subspaces in $\mathcal{H}$ is said to be a \emph{ $K$-fusion frame} for  $\mathcal{H}$  if there exist constants $0<A, B < \infty $ so that for every $f \in \mathcal H$ we have,
 \begin{equation}A\|K^*f\|^2~ \leq ~\sum_{i \in \mathcal I} w_i ^2\|P_{\mathcal W_i}f\|^2 ~\leq ~B\|f\|^2.\end{equation}
\end{defn}





\subsection{{Woven and $K$-Woven Frame}} In general in a sensor networking system, a frame can be characterized by signals. If there are two frames, having same characteristics, then in absence of a frame element from the first frame, still we are able to get an error free result on account of the replacement of the frame element of first frame by the frame element of second frame.

In this context basically one can think of the  intertwinedness between two sets of sensors, or in general between two frames, which leads to the idea of weaving frames. Weaving frames or woven frames were introduced by Bemrose et al. in \cite{BeCaGrLaLy16}. Later the concept of woven-ness has been characterized by Bhandari et al. in \cite{BhMu19} and characterization of weaving $K$-frames has been produced by Deepshikha et al. in \cite{DeVa18}.


\begin{defn} \cite{BhMu19}
	In $\mathcal H$, two frames $\lbrace f_i \rbrace_{i \in \mathcal I}$ and  $\lbrace g_i \rbrace_{i \in \mathcal I}$ are said to be {\it woven} if for every $\sigma \subset \mathcal I$,  $\lbrace f_i \rbrace_{i \in \sigma} \cup \lbrace g_i \rbrace_{i \in \sigma^c}$ also forms a frame for $\mathcal H$ and the associated frame operator for every weaving is defined as,
	$$S_{\mathcal F \mathcal G}f=\sum_{i \in \sigma}\langle f,f_i \rangle f_i + \sum_{i \in \sigma^c}\langle f,g_i \rangle g_i, ~~ \text{for}~ \text{all}~f \in \mathcal H.$$
\end{defn}

\begin{defn} \cite{DeVa18} A family of K-frames $\{\{\phi_{ij}\}_{j=1}^\infty: i \in[m]\} $ for $\mathcal{H}$ is said to be K-woven if there exist universal positive constants $ A, B $ such that for any partition $\{\sigma_i\}_{i\in[m]}$ of $\mathcal{I}$, the family $\bigcup_{i \in [m]} \{\phi_{ij}\}_{j\in\{\sigma_i\}}$ is a K-frame for $\mathcal{H}$ with lower and upper K-frame bounds A and B, respectively. Each family $\bigcup_{i \in [m]} \{\phi_{ij}\}_{j\in \sigma_i}$ is called a $K$-weaving.
\end{defn}

\begin{defn} \cite{GaVa17} Let $K \in \mathcal{L}(\mathcal{H})$ and consider two $K$-fusion frames $\{(\mathcal W_i, w_i)\}_{i \in \mathcal I}$, $\{(\mathcal V_i, v_i)\}_{i \in \mathcal I}$. Then they are said to be woven if there are universal constants $A, B$ so that for every $\sigma \subset \mathcal I$ and for every $f \in \mathcal H$ we have,
 \begin{equation}A\|K^*f\|^2~ \leq ~\sum_{i \in \sigma} w_i ^2\|P_{\mathcal W_i}f\|^2 + \sum_{i \in \sigma^c} v_i ^2\|P_{\mathcal V_i}f\|^2 ~\leq ~B\|f\|^2.\end{equation}
\end{defn}

The following result presents the woven-ness of $K$-fusion Bessel sequences.

\begin{prop} \cite {GaVa17} \label{Bessel} Let $ K \in \mathcal{L}(\mathcal{H})$ for which $\mathcal W_w = \{(\mathcal W_i, w_i)\}_{i \in \mathcal I}$ and $\mathcal V_v = \{(\mathcal V_i, v_i)\}_{i \in \mathcal I}$ be $ K$- fusion Bessel sequences in $\mathcal{H}$ with bounds $B_1, B_2$ respectively. Then for every $\sigma \subset \mathcal I$, the associated weaving between them also forms a $ K$- fusion Bessel sequence in $\mathcal{H}$ with the universal bound $B_1 + B_2$.
	
\end{prop}









The following Lemma provides a discussion regarding Moore-Penrose  pseudo-inverse. For detail discussion regarding the same we refer \cite{Ch08, Ka80}.


\begin{lem}\label{Moore} Let $\mathcal H$ and $\mathcal K$ be two Hilbert spaces and $T \in \mathcal L(\mathcal H,\mathcal K)$ be a closed range operator, then the followings hold:
	\begin{enumerate}
		\item $TT^{\dag}=P_{R(T)}$, $T^{\dag}T=P_{R(T^*)}$
		\item $\frac {\|f\|} {\|T^{\dag}\|} \leq \|T^*f\|$ for all $f \in T(\mathcal H)$.
		\item $TT^{\dag}T=T$, $T^{\dag}TT^{\dag}=T^{\dag}$, $(TT^{\dag})^*=TT^{\dag}$, $(T^{\dag}T)^*=T^{\dag}T$.
	\end{enumerate}
\end{lem}

\begin{lem}(\cite{Ga07, LiYaZh15}) \label{lem-projection} Suppose  $\mathcal H$ and $\mathcal K$ are two Hilbert spaces and $T \in \mathcal L(\mathcal H,\mathcal K)$. Consider $\mathcal W$ be a closed subspace of $\mathcal H$ and $\mathcal V$ be a closed subspace of $\mathcal K$. Then the following results are satisfied:
	\begin{enumerate}
		\item $P_{\mathcal W}T^*P_{\overline {T \mathcal W}}=P_{\mathcal W}T^*$.
		\item $P_{\mathcal W}T^*P_{\mathcal V}=P_{\mathcal W}T^*$ if and only if $T \mathcal W \subset \mathcal V$.
	\end{enumerate}
\end{lem}

Applying the foregoing Lemma we fabricate an analogous result.

\begin{lem} \label{analog} Let  $\mathcal H_1$, $\mathcal H_2$ be two Hilbert spaces and $T \in \mathcal L(\mathcal H_1, \mathcal H_2)$ be one-one, closed range operator. Suppose $\mathcal W$ is a closed subspace of $\mathcal H_1$ and $T(\mathcal W)$ is a closed subspace of $\mathcal H_2$. Then the following holds:
$$P_{T(\mathcal W)}T^{\dag *}P_{T^{\dag} {T (\mathcal W)}}= P_{T(\mathcal W)}T^{\dag *}P_{\mathcal W} = P_{T(\mathcal W)}T^{\dag *}.$$
\end{lem}

\section{Main Results}\label{Results}

We begin this section by providing two intertwining results on $K$-fusion frames between two separable Hilbert spaces.

\begin{lem}\label{lem_K_intertwine1}
Let $ K \in \mathcal{L}(\mathcal{H}_1)$ for which $\mathcal W_w = \{(\mathcal W_i, w_i)\}_{i \in \mathcal I}$  be a K- fusion frame for $\mathcal{H}_1$. Suppose $T \in \mathcal L(\mathcal H_1, \mathcal H_2)$ is a closed range operator with $T^{\dag} \overline {T(\mathcal W_i)} \subset \mathcal W_i$, for all $i \in \mathcal I$ and $\sum \limits_{i \in \mathcal I} w_i ^2 < \infty$. Then $\{(T\mathcal W_i, w_i)\}_{i \in \mathcal I}$ forms a $TKT^*$- fusion frame for $\mathcal{H}_2$.
\end{lem}

\begin{proof}
First we prove for all $i \in \mathcal I$, $T(\mathcal W_i)$ is a closed subspace in  $\mathcal H_2$. Since $T^{\dag} \overline {T(\mathcal W_i)} \subset \mathcal W_i$, then $T T^{\dag} \overline {T(\mathcal W_i)} \subset T(\mathcal W_i)$. But applying Lemma 2.5.2 of \cite{Ch08} we have $\left.T^{\dag}\right|_{R(T)}=T^* (TT^*)^{-1}$ and hence $\overline {T(\mathcal W_i)} \subset T(\mathcal W_i)$. Therefore, for every $i \in \mathcal I$, $T(\mathcal W_i)$ is a closed subspace in  $\mathcal H_2$. Since $\{(\mathcal W_i, w_i)\}_{i \in \mathcal I}$ is a $K$-fusion frame for $\mathcal{H}_1$, there exist $A, B>0$ so that for every $f \in \mathcal{H}_1$ we have,
\begin{equation}\label{K-f.f.} A \|K^*f\|^2 \leq \sum_{i \in \mathcal I} w_i ^2 \|P_{\mathcal W_i}f\|^2 \leq B \|f\|^2.\end{equation}
Again applying Lemma \ref{lem-projection} and using equation (\ref{K-f.f.}), for every $f \in \mathcal H_2$ we obtain,
\begin{eqnarray*}
\frac {A}{\|T\|^2} \|(TKT^*)^*f\|^2 \leq A \|K^*(T^*f)\|^2 &\leq & \sum \limits_{i \in \mathcal I} w_i^2 \|P_{\mathcal W_i} T^*f\|^2
\\ &= & \sum \limits_{i \in \mathcal I} w_i^2 \|P_{\mathcal W_i} T^* P_{\overline{T\mathcal W_i}}f\|^2
\\ &= & \sum \limits_{i \in \mathcal I} w_i^2 \|P_{\mathcal W_i} T^* P_{T\mathcal W_i}f\|^2
\\& \leq & \|T\|^2 \sum \limits_{i \in \mathcal I} w_i^2 \| P_{T\mathcal W_i}f\|^2
\end{eqnarray*}
and hence $\sum \limits_{i \in \mathcal I} w_i^2 \| P_{T\mathcal W_i}f\|^2 \geq \frac {A}{\|T\|^4} \|(TKT^*)^*f\|^2$.
Furthermore, since $\sum \limits_{i \in \mathcal I} w_i ^2 < \infty$,  for every $f \in \mathcal H_2$ we get,
$\sum \limits_{i \in \mathcal I} w_i^2 \| P_{T\mathcal W_i}f\|^2 \leq (\sum \limits_{i \in \mathcal I} w_i ^2)\|f\|^2.$
\end{proof}

\begin{lem}\label{lem_K_intertwine2}
 Let $\{(\mathcal W_i, w_i)\}_{i \in \mathcal I}$ be a weighted collection of closed subspaces in $\mathcal H_1$ and  $T \in \mathcal L(\mathcal H_1, \mathcal H_2)$ be one-one, closed range operator so that for some $K \in \mathcal  L(\mathcal H_2)$, $\{(T\mathcal W_i, w_i)\}_{i \in \mathcal I}$ be a $K$-fusion frame for $R(T)$. Then $\{(\mathcal W_i, \frac {w_i} {\|T\|})\}_{i \in \mathcal I}$ forms a $T^{\dag} K T$-fusion frame for $\mathcal H_1$.
\end{lem}

\begin{proof} Since $\{(T\mathcal W_i, w_i)\}_{i \in \mathcal I}$ is a $K$-fusion frame for $R(T)$, there exist $A, B > 0$ so that for every $h_2^{(1)} \in R(T)$ we have,
\begin{equation} \label{K-frame}
	A \|K^* h_2^{(1)}\|^2 \leq \sum \limits_{i \in \mathcal I} w_i ^2 \|P_{T \mathcal W_i} h_2^{(1)}\|^2 \leq B \|h_2^{(1)}\|^2.
	\end{equation}
	
Now since $T$ is one-one and $R(T)$ is closed, for every $h_1 \in \mathcal H_1$ there exists $h_2 \in \mathcal H_2$ so that $h_1=T^{*} h_2$ and  for every $h_2 \in \mathcal H_2$ we have $h_2 = h_2 ^{(1)} + h_2 ^{(2)}$, where $h_2 ^{(1)} \in R(T)$ and $h_2 ^{(2)} \in R(T)^{\perp}$.

Therefore, $ h_2 ^{(1)} = T^{*\dag}(h_1 - T^*h_2 ^{(2)}) = T^{*\dag} h_1$. Hence applying Lemma \ref{analog} we get,
\begin{eqnarray*}
\sum \limits_{i \in \mathcal I} w_i ^2 \|P_{T \mathcal W_i} h_2^{(1)}\|^2 = \sum \limits_{i \in \mathcal I} w_i ^2 \|P_{T \mathcal W_i} T^{\dag *} h_1\|^2 & = & \sum \limits_{i \in \mathcal I} w_i ^2 \|P_{T \mathcal W_i} T^{\dag *} P_{\mathcal W_i} h_1\|^2
\\ & \leq & \|T^{\dag}\|^2\sum \limits_{i \in \mathcal I} w_i ^2 \| P_{\mathcal W_i} h_1\|^2.
\end{eqnarray*}
Consequently, using equation (\ref{K-frame}) for every $h_1 \in \mathcal H_1$ we obtain,
\begin{eqnarray*}
\sum \limits_{i \in \mathcal I} \left (\frac {w_i} {\|T\|} \right )^2 \| P_{\mathcal W_i} h_1\|^2 & \geq & \frac {A} {\|T\|^2 \|T^{\dag}\|^2} \|(T^{\dag} K)^* h_1\|^2
\\ & \geq & \frac {A} {\|T\|^4 \|T^{\dag}\|^2} \|(T^{\dag} K T)^* h_1\|^2 .
\end{eqnarray*}	
Furthermore, applying Lemma \ref{lem-projection} and using equation (\ref{K-frame}) for every $h_1 \in \mathcal H_1$ we get,
\begin{eqnarray*}
\sum \limits_{i \in \mathcal I} \left (\frac {w_i} {\|T\|} \right )^2 \| P_{\mathcal W_i} h_1\|^2 & = & \sum \limits_{i \in \mathcal I} \left (\frac {w_i} {\|T\|} \right )^2 \| P_{\mathcal W_i} T^*h_2\|^2
\\ & = & \sum \limits_{i \in \mathcal I} \left (\frac {w_i} {\|T\|} \right )^2 \| P_{\mathcal W_i} T^* P_{T \mathcal W_i} h_2\|^2
\\ & \leq & \sum \limits_{i \in \mathcal I} w_i ^2 \| P_{T \mathcal W_i} h_2\|^2
\\ & = & \sum \limits_{i \in \mathcal I} w_i ^2 \| P_{T \mathcal W_i} (h_2^{(1)} + h_2^{(2)} )\|^2
\\ & = & \sum \limits_{i \in \mathcal I} w_i ^2 \| P_{T \mathcal W_i} h_2^{(1)}\|^2
\\ & \leq & B\|h_2^{(1)}\|^2
\\& \leq & B \|T^{\dag}\|^2 \|h_1\|^2.
\end{eqnarray*}	
Hence our assertion is tenable.
\end{proof}

As a consequence of Lemma \ref{lem_K_intertwine1} and \ref{lem_K_intertwine2}, the following two propositions show that $K$-woven-ness is preserved under bounded linear operators.

\begin{prop} \label{prop_K_intertwine1} Let $K \in \mathcal L(\mathcal H_1)$ for which $\mathcal W_w =\{(\mathcal W_i, w_i)\}_{i \in \mathcal{I}}$ and $\mathcal V_v =\{(\mathcal V_i, v_i)\}_{i \in \mathcal{I}}$ be $K$-fusion frames for $\mathcal{H}_1$. Further let us consider a closed range operator $T \in \mathcal{L}(\mathcal{H}_1,\mathcal{H}_2)$ with $T^{\dag} \overline {T(\mathcal W_i)} \subset \mathcal W_i$ and $T^{\dag} \overline {T(\mathcal V_i)} \subset \mathcal V_i$, for all $i \in \mathcal I$  for all $i \in \mathcal I$. Suppose $\mathcal W_w$ and $\mathcal V_v$ are weaving  K-fusion frames for $\mathcal{H}_1$, then $\{(T \mathcal W_i, w_i)\}_{i \in \mathcal{I}}$ and $\{(T \mathcal V_i, v_i)\}_{i \in \mathcal{I}}$ are weaving $TKT^*$-fusion frames for $\mathcal{H}_2$.
\end{prop}
\begin{proof}
Applying Lemma \ref{lem_K_intertwine1}, our assertion is tenable.
\end{proof}

\begin{prop}\label{prop_K_intertwine2}
Let $\{(\mathcal W_i, w_i)\}_{i \in \mathcal{I}}$ and $\{(\mathcal V_i, v_i)\}_{i \in \mathcal{I}}$ be two weighted collections of closed subspaces in $\mathcal{H}_1$. Suppose $T \in \mathcal L(\mathcal H_1, \mathcal H_2)$ to be one-one, closed range operator so that for some $K \in \mathcal  L(\mathcal H_2)$, $\{(T \mathcal W_i, w_i)\}_{i \in \mathcal{I}}$ and $\{(T \mathcal V_i, v_i)\}_{i \in \mathcal{I}}$  are weaving $K$-fusion frames for $R(T)$ with the universal bounds $A, B$. Then $\{(\mathcal W_i, \frac {w_i} {\|T\|})\}_{i \in \mathcal{I}}$ and $\{(\mathcal V_i, \frac {v_i} {\|T\|})\}_{i \in \mathcal{I}}$ are  weaving $T^{\dag} K T$-fusion frames for $\mathcal H_1$ with the universal bounds $\frac {A} {\|T\|^4 \|T\|^2}$, $B \|T^{\dag}\|^2$.
\end{prop}

\begin{proof} The proof will be followed from Lemma \ref{Moore} and Lemma \ref{lem_K_intertwine2}.
\end{proof}

In the following result we discuss images of weaving fusion frames under bounded, linear operator preserve their woven-ness with respect to the said operator.

\begin{prop}\label{equivalency1}
	Let $\{(\mathcal W_i, w_i)\}_{i \in \mathcal I}$ and $\{(\mathcal V_i, v_i)\}_{i \in \mathcal I}$ be weaving fusion frames for $\mathcal H$. Then for every $K \in \mathcal L(\mathcal H)$, $\{(\overline{K\mathcal W_i}, w_i)\}_{i \in \mathcal I}$ and $\{(\overline{K\mathcal V_i}, v_i)\}_{i \in \mathcal I}$ are weaving $K$-fusion frames for $\mathcal H$.
\end{prop}

\begin{proof}
 Let $\{(\mathcal W_i, w_i)\}_{i \in \mathcal I}$ and $\{(\mathcal V_i, v_i)\}_{i \in \mathcal I}$ be weaving fusion frames for $\mathcal H$ with the universal bounds $A, B$. Then for every $\sigma \subset \mathcal I$ and  $f \in \mathcal H$ we have,
\begin{equation}\label{K-f.f. 1} A \|f\|^2 \leq \sum_{i \in \sigma} w_i ^2 \|P_{\mathcal W_i}f\|^2 + \sum_{i \in \sigma ^c} v_i ^2 \|P_{\mathcal V_i}f\|^2 \leq B \|f\|^2.\end{equation}	
Therefore,  using equation (\ref{K-f.f. 1}) and applying Lemma \ref{lem-projection}, for every $K \in \mathcal L(\mathcal H)$, $\sigma \subset \mathcal I$ and  $f \in \mathcal H$ we obtain,
$$\sum_{i \in \sigma} w_i ^2 \|P_{\overline{K\mathcal W_i}}f\|^2 + \sum_{i \in \sigma ^c} v_i ^2 \|P_{\overline{K \mathcal V_i}}f\|^2 \geq \frac {A} {\|K\|^2} \|K^* f\|^2.$$
	The universal upper bound of the respective weaving will achieved by Proposition \ref{Bessel}.
\end{proof}

Next result provides a characterization of weaving fusion frames by means of weaving $K$-fusion frames and conversely.

\begin{prop} \label{equivalency2}
Let $K \in \mathcal L(\mathcal H)$ and consider two weighted collections $\mathcal W_w$, $\mathcal V_v$ of closed subspaces of $\mathcal H$. Then
\begin{enumerate}[label=(\roman*)]

\item $\mathcal W_w$ and $\mathcal V_v$ are weaving $K$-fusion frames for $\mathcal H$ whenever they form weaving fusion frames for  $\mathcal H$.

\item if  $R(K)$ is closed, then $\mathcal W_w$ and $\mathcal V_v$ form weaving fusion frames for $R(K)$ whenever they are weaving $K$-fusion frames for $R(K)$.
\end{enumerate}
\end{prop}

\begin{proof}
\begin{enumerate}[label=(\roman*)]
\item Let $\mathcal W_w$ and $\mathcal V_v$ be weaving fusion frames for $\mathcal H$ with the universal bounds $A, B$. Then for every $\sigma \subset \mathcal I$ and  $f \in \mathcal H$ we get,
$$\frac {A} {\|K\|^2} \|K^*f\|^2 \leq A \|f\|^2 \leq \sum_{i \in \sigma} w_i ^2 \|P_{\mathcal W_i}f\|^2 + \sum_{i \in \sigma ^c} v_i ^2 \|P_{\mathcal V_i}f\|^2 \leq B \|f\|^2.$$

\item Suppose $\mathcal W_w$ and $\mathcal V_v$ are weaving $K$-fusion frames for $R(K)$ with the universal bounds $C, D$. Then for every $\sigma \subset \mathcal I$ and  $f \in \mathcal H$ we have,
\begin{equation} \label{K-f.f. 2}
C \|K^*f\|^2  \leq \sum_{i \in \sigma} w_i ^2 \|P_{\mathcal W_i}f\|^2 + \sum_{i \in \sigma ^c} v_i ^2 \|P_{\mathcal V_i}f\|^2 \leq D \|f\|^2.
\end{equation}
Again using closed range property for every $f \in R(K)$ we have, $ \|K^*f\|^2 \geq \frac {1} {\|K^{\dag}\|^2} \|f\|^2$. Therefore, using equation (\ref{K-f.f. 2}) we obtain,
$$\frac {C} {\|K^{\dag}\|^2} \|f\|^2 \leq \sum_{i \in \sigma} w_i ^2 \|P_{\mathcal W_i}f\|^2 + \sum_{i \in \sigma ^c} v_i ^2 \|P_{\mathcal V_i}f\|^2 \leq D \|f\|^2.$$
\end{enumerate}
\end{proof}

In the following results we discuss stability of woven-ness of $K$-fusion frames under perturbation and erasures. Analogous erasure result for frame can be observed in \cite{CaLy15}.

\begin{thm} \label{pert}
	Let $T, K \in \mathcal{L}(\mathcal H)$ with $K$ has closed range  and suppose for every $f \in \mathcal H$ we have, $\|(T^* - K^*)f\| \leq \alpha_1 \|T^*f\| + \alpha_2 \|K^*f\| + \alpha_3 \|f\|$, for some $\alpha_1, \alpha_2, \alpha_3  \in (0,1)$. Then $\{(\mathcal W_i, w_i)\}_{i \in \mathcal{I}}$ and $\{(\mathcal V_i, v_i)\}_{i \in \mathcal{I}}$ are weaving $T$-fusion frames if  they are weaving $K$-fusion frames for $R(K)$.
\end{thm}

\begin{proof}
	Let $\{(\mathcal W_i, w_i)\}_{i \in \mathcal{I}}$ and $\{(\mathcal V_i, v_i)\}_{i \in \mathcal{I}}$ be weaving $K$-fusion frames with the universal bounds $A, B$. Then for every $\sigma \subset \mathcal{I}$ and every $f \in R(K)$ we have,
	\begin{eqnarray} \label{paley}
~~~~~~A\|K^*f\|^2 \leq \sum_{i \in \sigma} w_i ^2 \|P_{\mathcal W_i}f\|^2 + \sum_{i \in \sigma ^c} v_i ^2 \|P_{\mathcal V_i}f\|^2  \leq B \|f\|^2.
	\end{eqnarray}
	
Again for every $f \in \mathcal{H}$ we have, $\|K^*f\| \geq  \|T^*f\| - \|(T^* - K^*)f\|$  and hence applying closed range property of $K$ (see Lemma \ref{Moore}) and employing given perturbation condition for every $f \in R(K)$ we obtain,
	$$(1 - \alpha_1) \|T^*f\| \leq (1 + \alpha_2 + \alpha_3 \|K^{\dag}\|) \|K^*f\|.$$
	Therefore, using  equation (\ref{paley}), for every $f \in R(K)$ and every $\sigma \subset \mathcal{I} $ we obtain,
	\begin{eqnarray*}
 A \left (\frac{1 - \alpha_1} { 1+ \alpha_2 + \alpha_3 \|K^{\dag}\|}\right)^2 \|T^*f\|^2
\leq  \sum_{i \in \sigma} w_i ^2 \|P_{\mathcal W_i}f\|^2 + \sum_{i \in \sigma ^c} v_i ^2 \|P_{\mathcal V_i}f\|^2
 \leq  B\|f\|^2.
	\end{eqnarray*}
	
\end{proof}

\begin{cor}
Let $T, K \in \mathcal{L}(\mathcal H)$ and suppose $\alpha_1, \alpha_2 \in (0,1)$ so that for every $f \in \mathcal H$ we have,
$\|T^* f - K^*f\| \leq \alpha_1 \|T^* f\| + \alpha_2 \|K^*f\|$. Then  $\{(\mathcal W_i, w_i)\}_{i \in \mathcal{I}}$ and $\{(\mathcal V_i, v_i)\}_{i \in \mathcal{I}}$ are $T$-woven if and only if they are $K$-woven.
\end{cor}

\begin{thm}\label{Thm:erasure1}
Let $K \in \mathcal L(\mathcal H_1)$ for which $\{(\mathcal W_i, w_i)\}_{i \in \mathcal{I}}$ and $\{(\mathcal V_i, v_i)\}_{i \in \mathcal{I}}$ be weaving $K$-fusion frames for $\mathcal H_1$  with universal lower bound $A$ and suppose $T \in \mathcal{L}(\mathcal{H}_1,\mathcal{H}_2)$ with $T^{\dag} \overline {T(\mathcal W_i)} \subset \mathcal W_i$ and $T^{\dag} \overline {T(\mathcal V_i)} \subset \mathcal V_i$ for all $i \in \mathcal I$ . Let us assume $ \mathcal{J} \subset \mathcal{I}$ and $ 0<C<\frac{A}{\|T\|^2}$ so that for every  $ f \in \mathcal{H}_2 $
	\begin{eqnarray} \label{inequa K-f.f.}
	\sum\limits_{i \in \mathcal J} w_i ^2 \|P_{T\mathcal W_i }\|^2 \leq C \|TK^*T^*f\|^2.
	\end{eqnarray}
	Then $\{(T\mathcal W_i, w_i)\}_{i \in \mathcal I \setminus \mathcal J}$ and $\{(T\mathcal V_i, v_i)\}_{i \in \mathcal I \setminus \mathcal J}$ form weaving $TKT^*$-fusion frames for $\mathcal H_2$ .
\end{thm}

\begin{proof} Since $\{(\mathcal W_i, w_i)\}_{i \in \mathcal{I}}$ and $\{(\mathcal V_i, v_i)\}_{i \in \mathcal{I}}$ are weaving $K$-fusion frames for $ \mathcal{H}_1 $, then by Lemma \ref{lem_K_intertwine1} and Proposition \ref{prop_K_intertwine1},  $\{(T\mathcal W_i, w_i)\}_{i \in \mathcal{I}}$ and $\{(T\mathcal V_i, v_i)\}_{i \in \mathcal{I}}$ form weaving $TKT^*$-fusion frames for $\mathcal H_2$  with universal lower bound $\frac{A}{\|T\|^2}$ in $\mathcal H_2$. Therefore, applying equation (\ref{inequa K-f.f.}), for every $\sigma \subset \mathcal{I} \setminus \mathcal{J}$ and for every $f \in \mathcal H_2$ we obtain,
	\begin{eqnarray*}
		\sum\limits_{i \in \sigma} w_i ^2 \|P_{T\mathcal W_i }\|^2 +  \sum\limits_{i \in \sigma ^c} v_i ^2 \|P_{T\mathcal V_i }\|^2 & =& \sum\limits_{i \in \sigma \cup \mathcal{J}} w_i ^2 \|P_{T\mathcal W_i }\|^2 + \sum\limits_{i \in \sigma^c}  v_i ^2 \|P_{T\mathcal V_i }\|^2 - \sum\limits_{i \in \mathcal{J}} w_i ^2 \|P_{T\mathcal W_i }\|^2\\
		& \geq & \frac{A}{\|T\|^2} \|(TKT^*)^*f\|^2 - C\|(TKT^*)^*f\|^2 \\
		& = & \left(\frac{A}{\|T\|^2}-C\right)\|(TKT^*)^*f\|^2,
	\end{eqnarray*}
	where $\sigma^c$ is the complement of $\sigma$ in $\mathcal I \setminus \mathcal J$.
	
	The universal upper bound will be followed by Proposition \ref{Bessel}.	
\end{proof}


By choosing $\mathcal H_1=\mathcal H_2$ and $T = I$, we obtain the following result.

\begin{cor}
Let $\{(\mathcal W_i, w_i)\}_{i \in \mathcal{I}}$ and $\{(\mathcal V_i, v_i)\}_{i \in \mathcal{I}}$ be weaving $K$-fusion frames for $ \mathcal{H}$ with the universal bounds A, B. Let us consider $ \mathcal{J} \subset \mathcal{I}$ and $0<C<A$ so that for every $f\in \mathcal{H}$,
	$$ \sum\limits_{i \in \mathcal{J}}w_i ^2 \|P_{\mathcal W_i }\|^2 \leq C \|K^*f\|^2,$$
	then $\{(\mathcal W_i, w_i)\}_{i \in \mathcal I \setminus \mathcal J}$ and $\{(\mathcal V_i, v_i)\}_{i \in \mathcal I \setminus \mathcal J}$ are weaving $K$-fusion frames  $ \mathcal{H}$  with the universal bounds $(A- C), B$.
\end{cor}


Using Proposition \ref{prop_K_intertwine2}, we get the following result analogous to Theorem \ref{Thm:erasure1}.
\begin{thm}\label{Thm:erasure2}
 Let $\{(\mathcal W_i, w_i)\}_{i \in \mathcal{I}}$ and $\{(\mathcal V_i, v_i)\}_{i \in \mathcal{I}}$ be two weighted collections of closed subspaces in $\mathcal{H}_1$ and $K \in \mathcal  L(\mathcal H_2)$. Suppose $T \in \mathcal L(\mathcal H_1, \mathcal H_2)$ is one-one, closed range operator so that $\{(T \mathcal W_i, w_i)\}_{i \in \mathcal{I}}$ and $\{(T \mathcal V_i, v_i)\}_{i \in \mathcal{I}}$  are weaving $K$-fusion frames for $R(T)$ with the universal lower bound $A$. Further suppose  $ \mathcal{J} \subset \mathcal{I}$ and $ 0<C<\frac{A}{\|T\|^4 \|T^{\dag}\|^2}$ so that for every  $ f \in \mathcal{H}_1 $
	\begin{eqnarray} \label{inequa}
	\sum\limits_{i \in \mathcal J} \left ( \frac {w_i} {\|T\|} \right)^2 \|P_{\mathcal W_i} f\|^2 \leq C \|(T^\dag K T)^*f\|^2.
	\end{eqnarray}
	Then, $\{(\mathcal W_i, \frac {w_i} {\|T\|})\}_{i\in \mathcal I \setminus \mathcal J}$ and $\{(\mathcal V_i, \frac {v_i} {\|T\|})\}_{i\in \mathcal I \setminus \mathcal J}$ are weaving $T^{\dag} K T$-fusion frames for $\mathcal{H}_1$.
\end{thm}

	
	

{\bf{Acknowledgment}}

The first author acknowledges the fiscal support of MHRD, Government of India.

\end{document}